\documentclass[11pt,english]{article}

\usepackage{amsfonts,amsmath,amsthm,amscd,amssymb,latexsym,cite,verbatim,texdraw,floatflt,
caption2,pb-diagram}

\usepackage{psfig}

\theoremstyle{plain}
\newtheorem{theorem}{Theorem}[section]

\newtheorem{definition}{Definition}[section]
\theoremstyle{definition}
\newtheorem{example}{Example}[section]
\newtheorem{remark}{Remark}[section]
\newcommand{\keywords}{\textbf{Key words. }\medskip}
\newcommand{\subjclass}{\textbf{MSC 2010. }\medskip}
\renewcommand{\abstract}{\textbf{Abstract. }\medskip}

\numberwithin{equation}{section}

\begin{document}

\title{Asymptotic behaviour of functionals of cyclical long-range dependent random fields\thanks{This work was partly supported by La Trobe University Research Grant "Stochastic Approximation in Finance and Signal Processing" and the Swedish Institute grant SI-01424/2007.}}

\author{Boris Klykavka, Andriy Olenko,  Matthew Vicendese}



\date{This is an Author's Accepted Manuscript of an article published in the
Ukrainian Mathematical Bulletin, Vol.~9, No. 2. (2012), 191--208.  and republished by Journal of Mathematical Sciences. Vol.~187, No. 1. (2012), 35--48. The final publication is available at link.springer.com. [DOI:10.1007/s10958-012-1047-1]}

\maketitle

\begin{abstract}
 Long-range dependent random fields with spectral densities which are unbounded at some frequencies are investigated.  We demonstrate new examples of covariance functions which do not exhibit  regular varying asymptotic  behaviour at infinity. However, variances of averaged functionals of these fields are regularly varying.  Limit theorems for weighted functionals of cyclical long-range dependent fields are obtained.  The order of normalizing constants and relations between the weight functions and singularities in non-degenerative asymptotics are discussed.
\end{abstract}



\section{Introduction}

For a homogeneous isotropic finite-variance random field $\xi(x),\ x\in \mathbb R^n,$ the most frequently admitted definition of long-range dependence is field's nonintegrable
covariance function
$\mathsf B_n(|x|)={\rm cov}(\xi(y+x),\xi(y)),$ i.e.
 \begin{equation}\label{long}\int_{\mathbb R^n} \mathsf B_n(|x|)dx=+\infty\,. \end{equation}
This is often wrongly thought to be due to a singularity of field's spectral density at zero.
However,  singularities of the spectral density at non-zero points also imply (\ref{long}).
In particular, models with singularities at non-zero frequencies are of great importance in time series analysis.
Many time series show cyclical/seasonal evolutions with peaks in spectral densities whose locations define periods of the cycles,
see \cite{anh,rob}.

Among the extensive literature on long-range dependence, relatively few publications are devoted to cyclical long-range dependent processes.

The case when the summands are some functionals of a long-range dependent Gaussian process is of great importance  in the theory of limit theorems for sums of dependent random variables.
It is well-known that, compared with Donsker-Prohorov results, long-range dependent summands can produce non-Gaussian limits and normalizing coefficients that are different to $n^{-1/2}.$

 Limit theorems for functionals of long-range dependent Gaussian processes were investigated by M. Rosenblatt \cite{ros1}, Dobrushin and Major \cite{dob1,dob}, Taqqu \cite{taq1,taq2}, Giraitis and Surgailis \cite{gir1,gir2}, Oppenheim and Viano  \cite{dou,ope,ope1} and others.  For multidimensional results of this type see \cite{leo1,leo2, ole2, yad}.

All  mentioned  publications were focused on the  Donsker line \;\;\;\;\;
$\int_{0}^{rt}H_k(\xi(x))dx,$ $ t\in[0,1]$ ($\sum_{m=0}^{[Nt]}H_k(\xi(m))$ for the discrete-time case), where $H_k(\cdot)$ is the $k$-th Hermite polynomial with the leading coefficient~1.

In the classical linear case $\int_{0}^{rt}\xi(x)dx,$ $H_1(x)=x,$ the limit is not affected by cyclicity at all.
However, it was shown that for non-linear functionals with $H_k(\cdot),\ k\ge 2,$  the cyclical behaviour can play a role, see \cite{dou,haye,ope}.

The aim of this paper is to investigate the limit behaviour of the weighted linear functionals
$$\int_{\mathbb{R}^n}f_{rt}(x)\,\xi(x)\,dx,\  r\to \infty, $$
where $\xi(x),$ $x\in \mathbb{R}^n,$ is a Gaussian random field, and $f_{r}(x)$ is a non-random function.

In the paper we establish a limit theorem for random fields, which is also new for the one-dimensional case of stochastic processes.  The theorem shows that
for general schemes, in contrast to the Donsker line, the cyclical effects play a role even for the linear case $H_1(x)=x.$
It is also shown that the limit is degenerated if the type of weight function does not match the type of the singularity.

The rest of the paper is organized as follows.  Section 2 gives some new examples of random fields which covariance functions do not
exhibit regular asymptotic behaviour at infinity, but which spectral densities are regularly varying at zero.
These examples motivate the use of weighted functionals for Abelian-Tauberian and limit theorems. All calculations in section 2 have been verified by {\it Maple 13,} Maplesoft.
Section 3 introduces a class of isotropic random fields with spectral singularities at non-zero frequencies.
In Section 4 we establish  a limit theorem. The case of degenerated limits is discussed in the next section.
Final remarks and and  particular cases are presented in Section 6.

In what follows we use the symbol $C$ with subscripts to denote constants which are not important for our discussion. Moreover, the same
symbol $C_i$ may be used for different constants appearing in the same proof.

\section{Motivating examples}

One of the main tools to obtain various asymptotic  results for     random processes is Abelian and Tauberian theorems. These theorems give relations of the asymptotic behaviour of covariance functions at infinity and spectral densities at zero.

Let $\mathbb{R}^n$ be a Euclidean space of dimension $n \geq 1.$ Furthermore, let $\xi(x),\,x \in \mathbb R^n $ be a real-valued
measurable mean-square continuous homogeneous isotropic Gaussian random field  with zero mean and the covariance function
 $\mathsf B_n(r) = \mathsf B_n (|x|)$ and  the spectral function $\mathbf{\Phi}(u), u\in \mathbb R^n, $ (see \cite{leo2,yad}).  If there is a function $\varphi(\lambda),\,\lambda\in[0;+\infty)$
  such that
\[\lambda^{n-1}\varphi(\lambda)\in L_1([0,\infty)),\, \mathbf{\Phi}(u)=\mathbf{\Phi}(|u|)=
\frac{2\pi^{n/2}}{\Gamma(n/2)}\int_0^{|u|} z^{n-1}\varphi(z)dz,\]
then $\varphi(\cdot)$ is called the isotropic spectral density of the field $\xi(x).$

To investigate asymptotic properties of random field numerous publications  (see, for example \cite{leo1,leo2,ole1,olkly,yad}) use variances of the functionals (instead of the covariance function $\mathsf  B_n(r)$)
\begin{equation}      \label{g001}
    \int_{v(r)}\xi(x)dx \quad \textrm{and} \quad \int_{s(r)}\xi(x)dm(x),
\end{equation}
where  $v(r):=\{x\in R^n: |x|\le r  \},$  $s(r):=\{x\in  R^n: |x|= r  \},$
$m(\cdot)$ is the Lebesgue measure on the sphere $s_{n-1}(r).$    Namely, they establish
and use  relations of asymptotic behaviour    of spectral densities $\varphi(\lambda)$   and the functions
$$
  l_n (r) ={\rm Var} \left[ \int_{s(r)} \xi(t) d m (t) \right]=
  \frac{(2 \pi)^n r^{2n}}{r^2}\int_0^{\infty}
  \frac{J^2_{\frac{n-2}{2}} (r\lambda)}{ (r\lambda)^{n-2}} d\mathbf{\Phi} (\lambda), n\geq 2,
$$
$$
   b_n (r) = { \rm Var} \left[ \int_{v(r)} \xi(t) dt  \right]=(2\pi)^n r^{2n} \int_0^{\infty}
     \frac{J^2_{\frac{n}{2}} (r\lambda)}{ (r\lambda)^{n}} d\mathbf{\Phi} (\lambda),
$$
where $J_\nu (\cdot) $ is the Bessel function of the first kind of order $\nu \ge -\frac{1}{2}.$

First of all we provide some motivation for using $l_n(r)$ and $b_n(r)$ instead of $\mathsf B_n(r).$
It is done by new examples of spectral densities, covariance functions and the functions $b_n(r), $  which
\begin{itemize}
  \item  have close form representations;
  \item  show principal differences in asymptotic behaviour of functions $b_n(r)$ and  $ \mathsf B_n(r),$ and clarify the choice of $b_n(r).$
\end{itemize}
In all examples below $n=3.$  Therefore

$$
   \mathsf B_3(r)=C{}  \int_0^{+\infty}  \frac{J_\frac{1}{2}(r \lambda)}{\sqrt{r \lambda}}{}\lambda^2{}\varphi(\lambda){}d\lambda,
$$
$$
    b_3(r)=  C_1\,r^6\int_0^{+\infty}\frac{J^2_{\frac 32}(r \lambda )}{(r \lambda )^3 }{}\lambda^2{}\varphi(\lambda){}d\lambda.
$$

\begin{example}  \label{ex1}
Let
    $$\varphi( \lambda)= \left\{
     \begin{array}{ll}
         \frac{2+\cos( \lambda)}{ \lambda}, & 0\le  \lambda \le a; \\
       0, & \;\; \textrm{otherwise}.
     \end{array}
     \right.$$
It is obvious that $ \lambda^2 \varphi( \lambda) \in L([0,+\infty]). $
Then the corresponding covariance function is
$$\mathsf B_3(r)= C \left( \frac{2\cos(ar)-2}{r^2 (r^2-1)} - \frac{ \cos(ar)(2+\cos(a))-3}{r^2-1} - \frac{ \sin(ar)\sin(a)}{r(r^2-1)}\right),  $$
and the function $b_3(r)$   has the form
$$b_3(r) =-\frac{1}{48\pi a^4}\left( A(r)+rC(r)+r^2D(r)+r^3G(r)-72a^4r^4\right):=I_1(r),$$
where
$$A(r)=48+24\cos(a)-8a\sin(a)-2a^4 Ci(a) -4 a^2 \cos(a)+a^4 Ci(2ar+a)$$
 $$ +a^4 Ci(2ar-a)-24\cos(a)\cos^2(ar)+4a^3\sin(a)-a^4\ln(2r-1)$$
 $$-a^4\ln(2r+1)-48\cos^2(ar); $$
$$C(r)=-24a\cos(a)\sin(2ar)+8a\sin(a)\cos^2(ar)+4a^2\cos(a)\cos^2(ar)$$
$$-4a^3\sin(a)\cos^2(ar)+4a^2\sin(2ar)(2\sin(a)+a\cos(a))-48a\sin(2ar); $$
$$D(r)= -12a^4 (Ci(2ar-a)+Ci(2ar+a))+24a^2(\cos(a)+a^2 Ci(a))$$
$$-40 a^3 \sin(a) +12a^4 ln(4r^2-1)-4a^4+48a^2+32a^3\sin(a)\cos^2(ar);  $$
$$G(r)=16a^4(Ci(2ar-a)-Ci(2ar+a)+ln(2r+1)-ln(2r-1));$$
 $$ Ci(x)=\gamma+\ln(x)+\int_0^x\frac{\cos(t)-1}{t}dt,  \;\;\;
   Si(x)=\int_0^x\frac {\sin(t)}{t}dt,$$
and $\gamma$  is the Euler's  constant.

For $a=1$ plots of  $B_3(r),$  $b_3(r),$  and their normalized transformations  are shown on Fig. 1 and Fig. 2.

\begin{center}
\begin{minipage}{5cm}
 {\psfig{figure=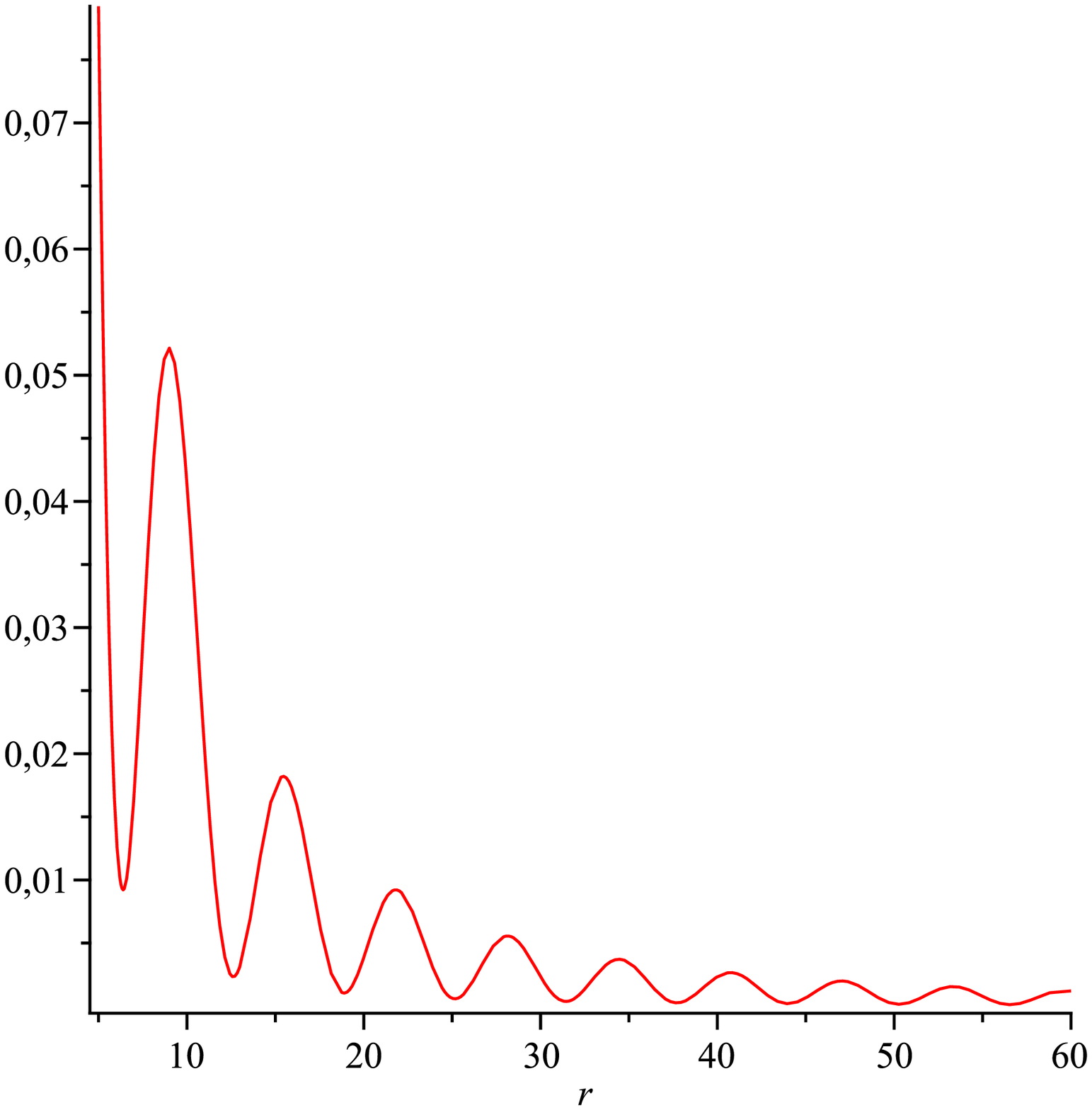 ,width=5cm}} \hspace{0.5cm}{  {\bf\rm Fig. 1a.} Plot of  $B_3(r)$} \end{minipage}
\begin{minipage}{5cm}
{\psfig{figure=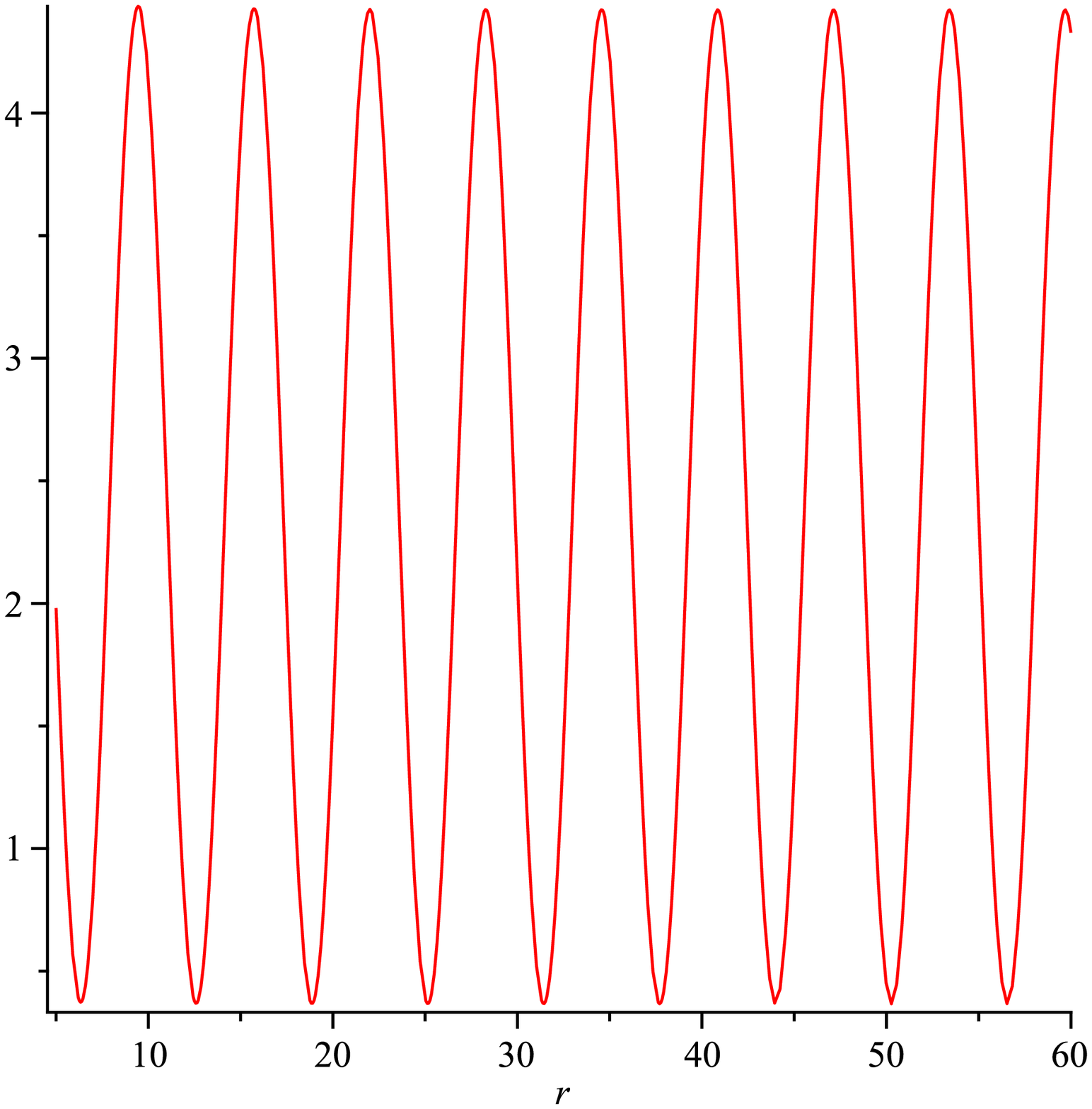,width=5cm}} \hspace{0.5cm}{ {\bf\rm Fig. 1b.} Plot of  $r^2 B_3(r)$} \end{minipage}
\end{center}

\begin{center}
\begin{minipage}{5cm}
 {\psfig{figure=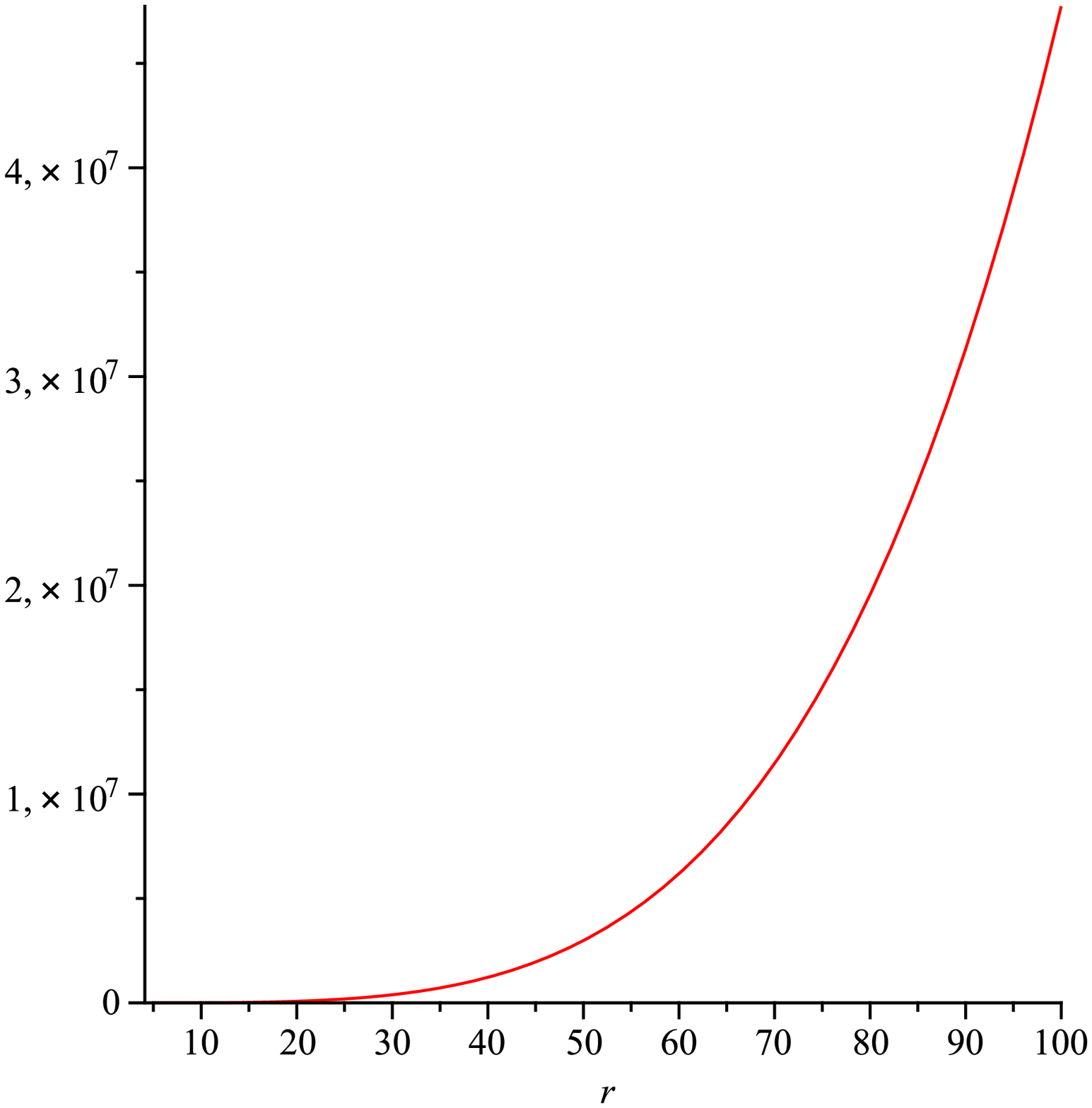 ,width=5cm}} \hspace{0.5cm}{ {\bf\rm Fig. 2a.}  Plot of  $b_3(r)$} \end{minipage}
\begin{minipage}{5cm}
{\psfig{figure=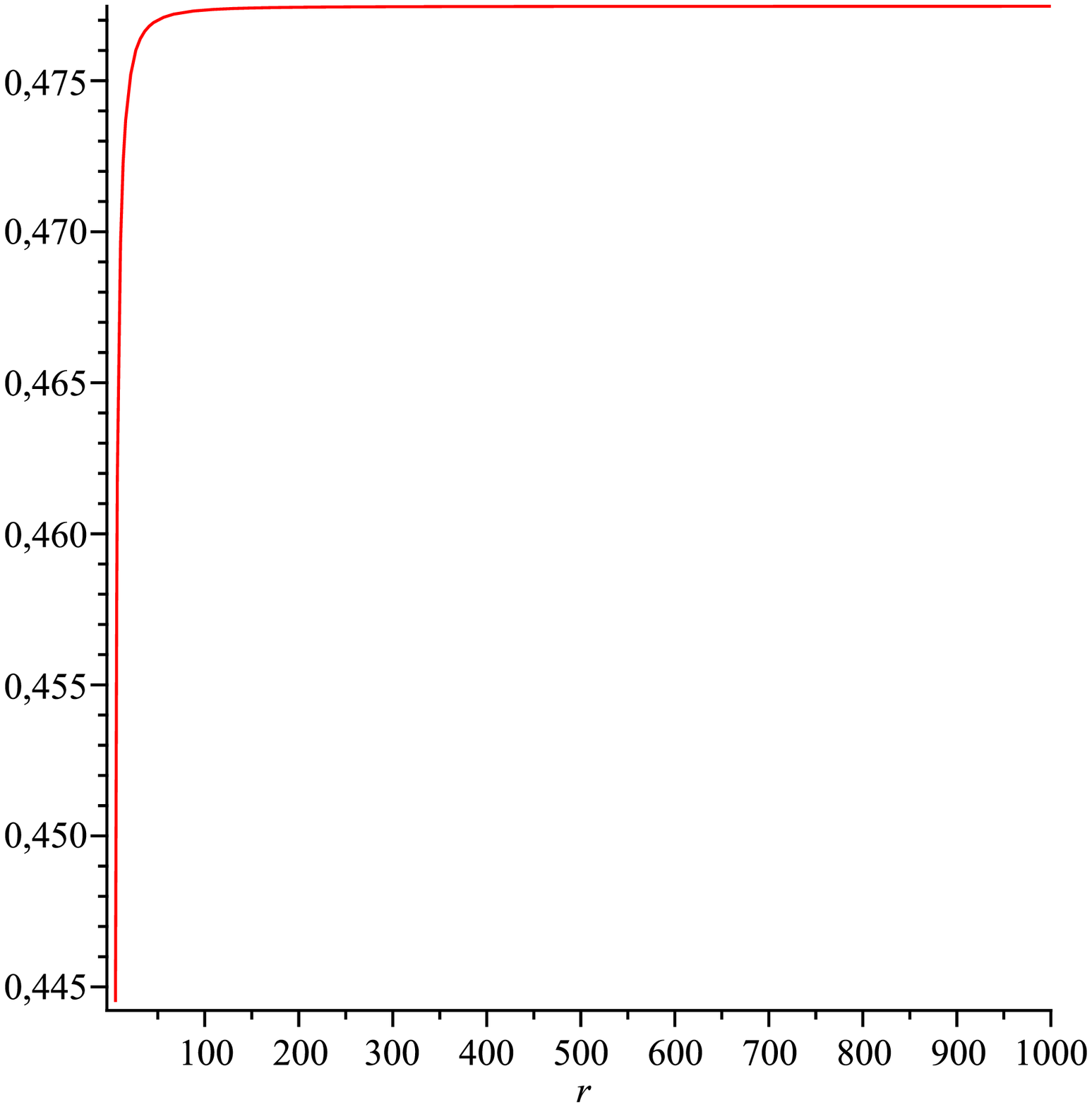,width=5cm}} \hspace{0.5cm}{ {\bf\rm Fig. 2b.} Plot of  $r^{-4} b_3(r)$} \end{minipage}
\end{center}

\end{example}

\begin{example}
Let
   $$\varphi(\lambda)= \left\{
     \begin{array}{ll}
       \frac{1+\sin(\lambda)}{\lambda}, & 0\le \lambda \le a; \\
       0, & \;\; \textrm{otherwise}.
     \end{array}
     \right.$$
Then the corresponding covariance function is
$$ B_3(r)= C \left( \frac{1-\cos(ar)(1+\sin(a))}{r^2-1}+\frac{ \sin(ar)\cos(a)}{r(r^2-1)}+\frac{ \cos(ar)-1}{r^2 (r^2-1)}\right),  $$
and the function  $b_3(r)$   has the form
$$b_3(r) =-\frac{1}{48\pi a^4}\left( A(r)+rC(r)+r^2D(r)+r^3G(r)-24a^4r^4\right),$$
where
 $$A(r)= 24 +  (4a^2-24) \sin(a)\cos^2(ar)+4a(a^2-2)\cos(a)\cos^2(ar)$$
 $$+a^4Si(2ar+a)+(8a-4a^3)\cos(a)-2a^4Si(a)-a^4\sin(2ar-a)$$
 $$+(24-4a^2)\sin(a)-24\cos^2(ar);   $$
$$D(r)=12a^4\left(Si(2ar-a)-Si(2ar+a)\right) +24a^2\left( a^2 Si(a)+\sin(a) \right)$$
$$+40 a^3\cos(a)-32 a^3\cos(a)\cos^2(ar)+24a^2;  $$
$$G(r)= -16a^4\left( Si(2ar+a)+Si(2ar-a)  \right).$$

For $a=1$ plots of  $B_3(r)  ,$  $b_3(r),$  and their normalized transformations  are shown on Fig.3 and Fig. 4.

\begin{center}
\begin{minipage}{5cm}
 {\psfig{figure=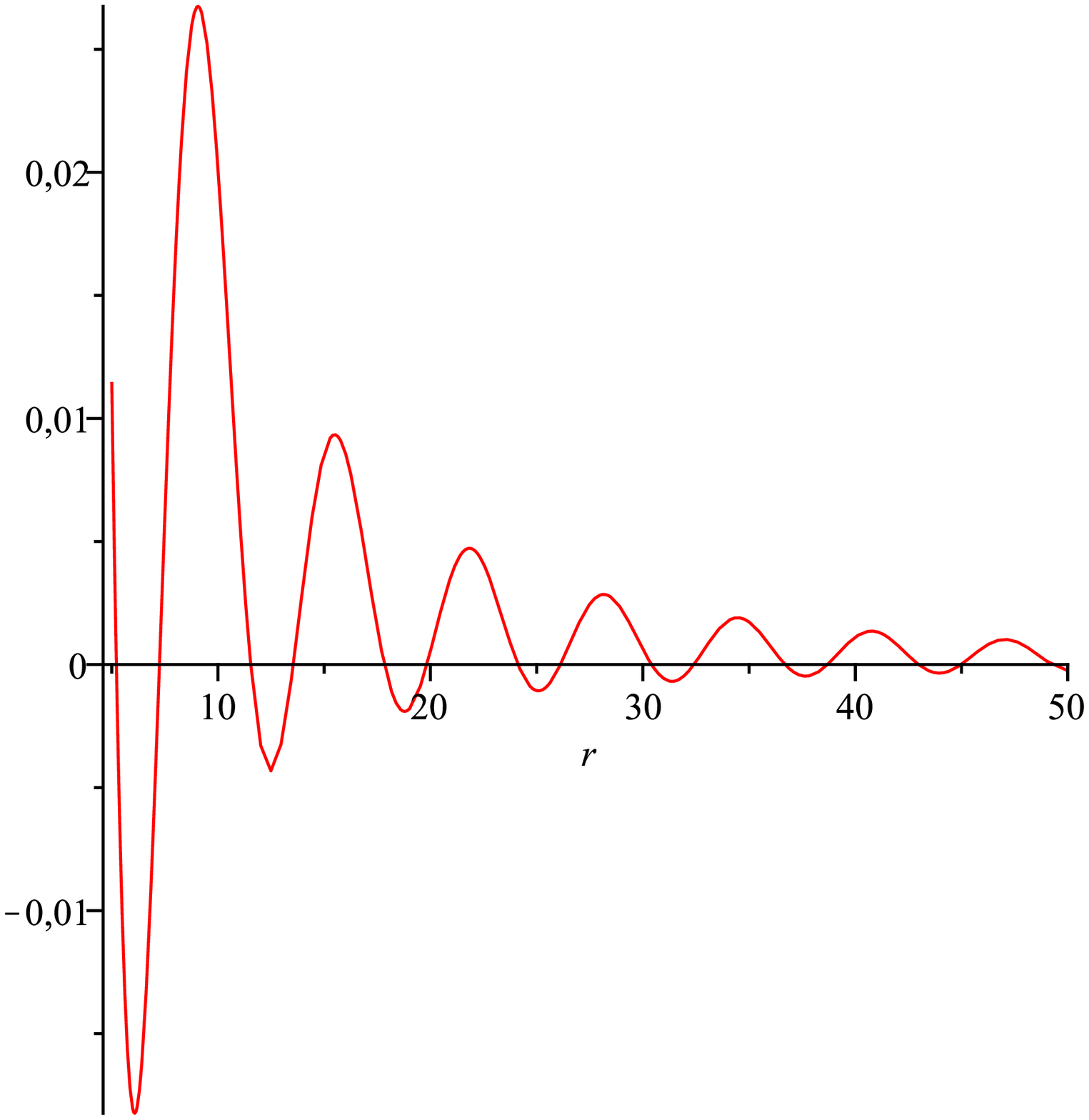 ,width=5cm}} \hspace{0.5cm}{ { \bf\rm Fig. 3a. } Plot of  $B_3(r)$} \end{minipage}
\begin{minipage}{5cm}
{\psfig{figure=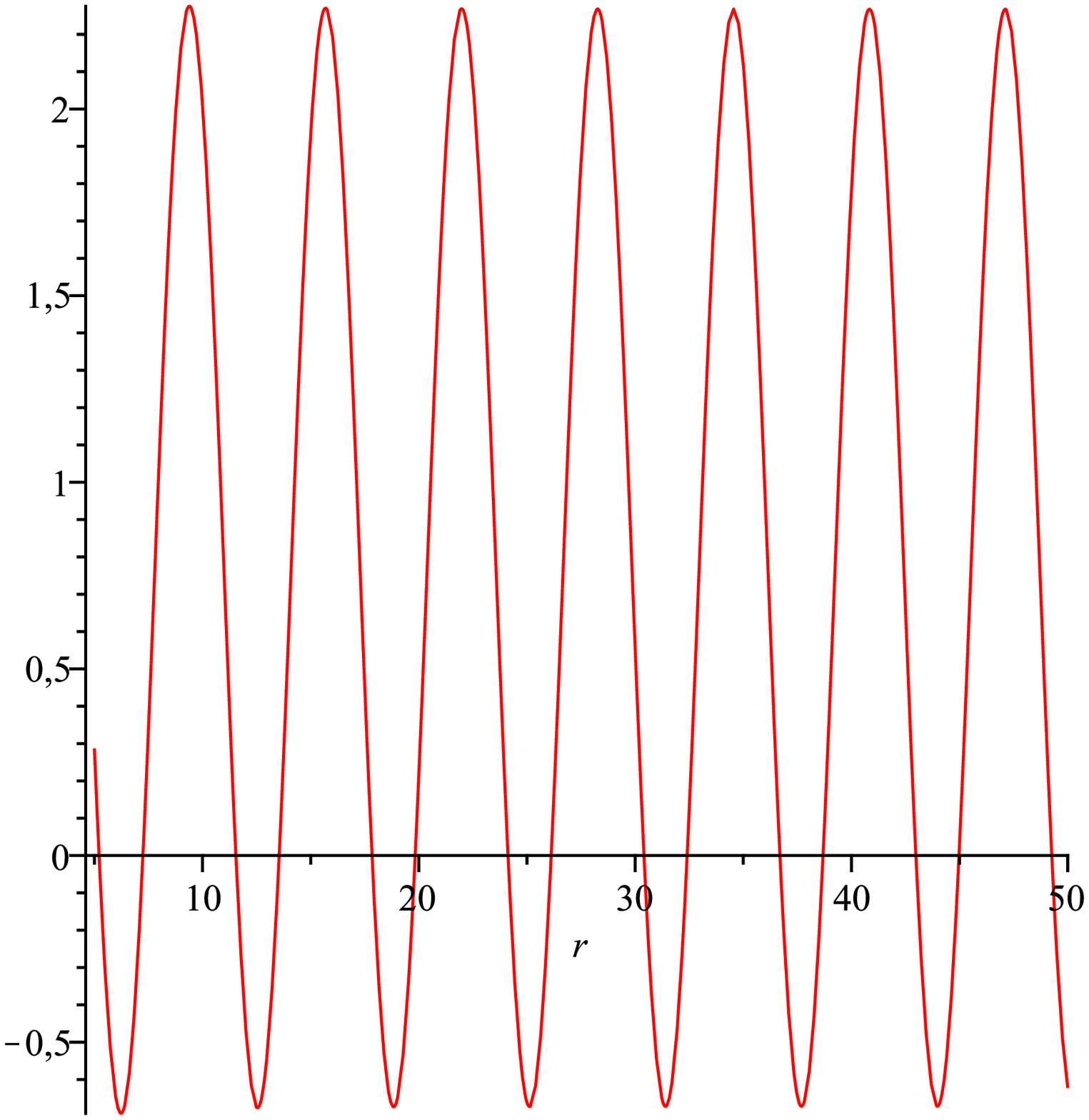,width=5cm}} \hspace{0.5cm}{ { \bf\rm Fig. 3b. } Plot of  $r^2 B_3(r)$} \end{minipage}
\end{center}

\begin{center}
\begin{minipage}{5cm}
 {\psfig{figure=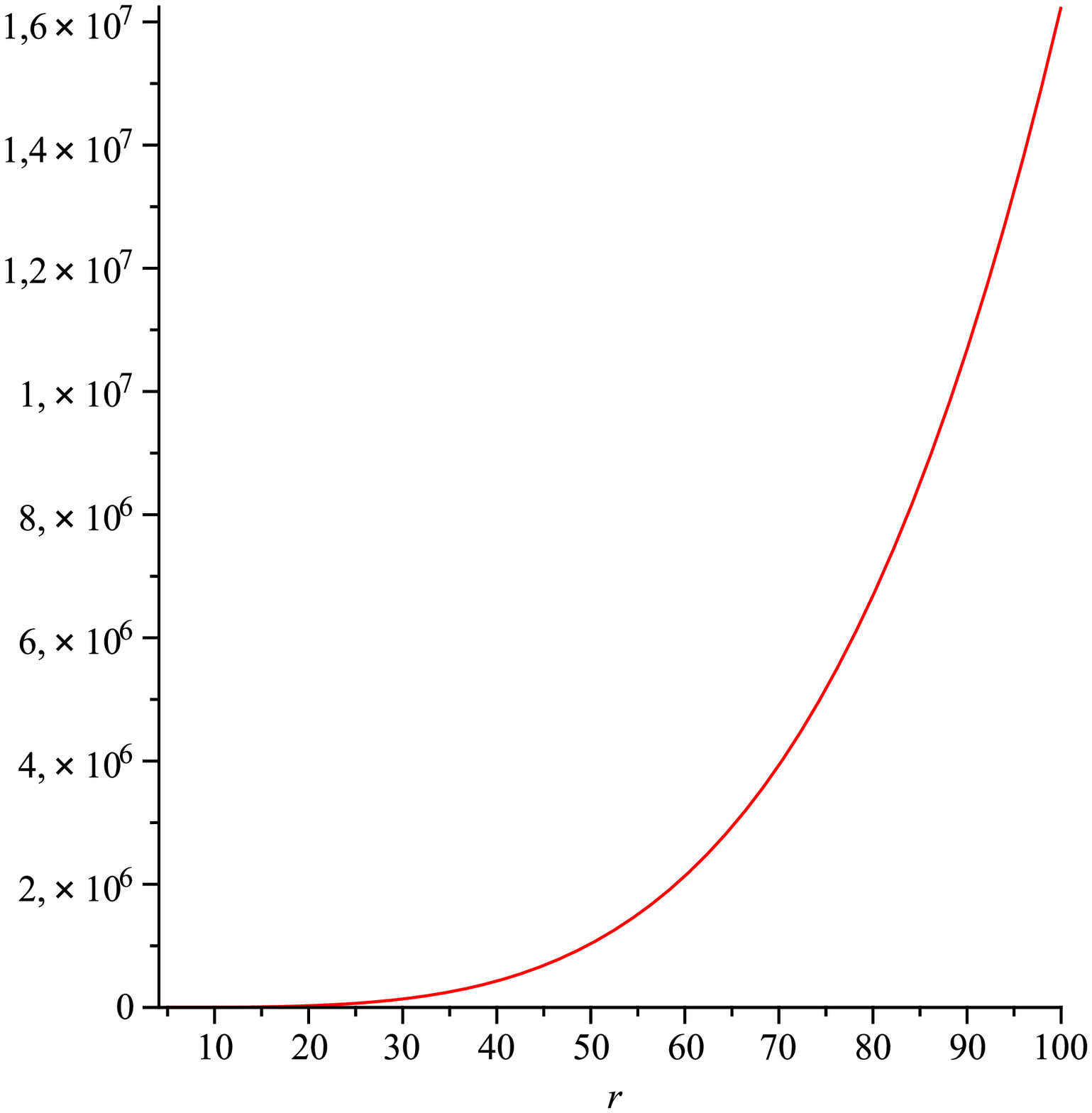 ,width=5cm}} \hspace{0.5cm}{ { \bf\rm Fig. 4a.} Plot of  $b_3(r)$} \end{minipage}
\begin{minipage}{5cm}
{\psfig{figure=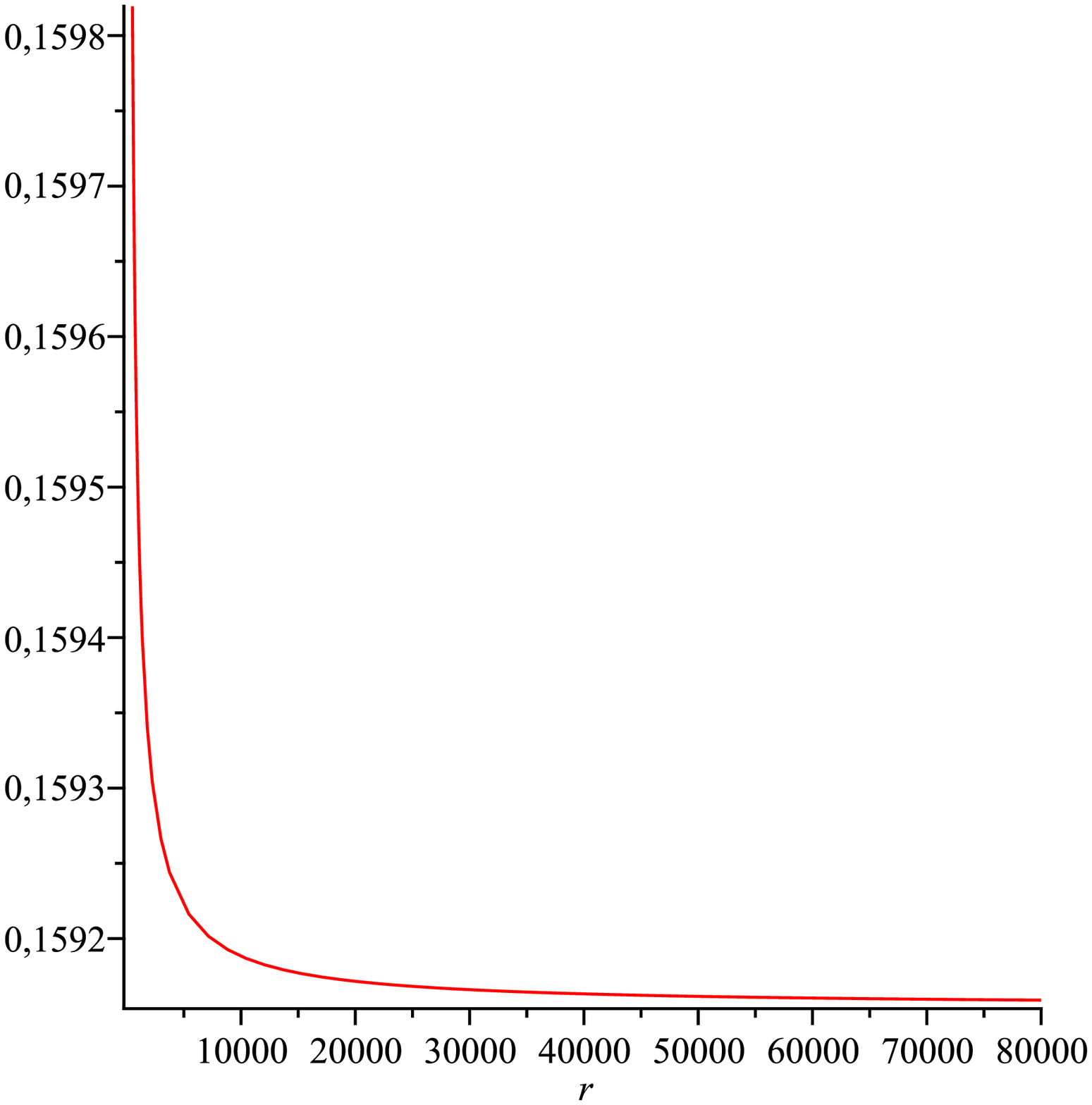,width=5cm}} \hspace{0.5cm}{ { \bf\rm Fig. 4b.} Plot of  $r^{-4} b_3(r)$} \end{minipage}
\end{center}

\end{example}

\begin{example}
Let
   $$\varphi(\lambda)= \left\{
     \begin{array}{ll}
        \frac{2(2+\cos(\lambda))}{\lambda}, & 0\le \lambda \le 1; \\
       \frac{1+\sin(\lambda)}{\lambda}, & 1< \lambda \le 2; \\
       0, & \;\; \textrm{otherwise}.
     \end{array}
     \right.$$
Then the corresponding covariance function is
 $$B_3(r) = -\frac{\sqrt{2}\left( A(r) r^2+C(r) r+D(r)\right) }{\sqrt{\pi}r^2 (r^2-1)}, $$
where
$$ A(r)= (3+2\cos(1))\cos(r)-\cos(r)\sin(1)$$
$$
+ 2\cos^2(r)(1+\sin(2)) -\sin(2)-7 ;$$
 $$C(r)= 2\sin(r)\sin(1) +\sin(r)\cos(1) -2\cos(2)\sin(r)\cos(r) ; $$
 $$D(r)=   -3\cos(r) -2\cos^2(r)+5. $$
And $b_3(r)$   has the form
$$  b_3(r)=C   \cdot  r^6 \left( \int_0^1 \frac{J_{\frac 3 2}^2(r x)}{(r x)^3}\, 2x \,(2+\cos(x))\, dx+\right. $$
 $$\left.+
 \int_1^2 \frac{J_{\frac 3 2}^2(r x)}{(r x)^3}\, x \,(1+\sin(x))\, dx  \right)=:2I_1(r)+I_2(r).
$$
  $I_1(r) $ is defined in Example \ref{ex1} and
 $$I_2(r) =\frac{1}{48\pi}\left( A(r)+rC(r)+r^2D(r)+r^3G(r)+72r^4\right),$$
where
$$A(r)= 24 -2Si(1)+
2Si(2) -30\cos^2(r)+ 6\cos^4(r)$$  $$+
\left(4\cos(2)-2\sin(2)\right)\cos^2(r)\sin^2(r)- Si(4r+2)+Si(4r-2)$$
$$+Si(2r+1)-Si(2r-1)
+\sin^2(r)\left(4\cos(1)+20\sin(1)\right);$$
$$C(r)=  \sin(r)\cos^3(r)\left(8 \sin(2) +16\cos(2)+24    \right)$$
$$  -\sin(r)\cos(r)\left(60+ 4 \sin(2) +16\cos(1) +40\sin(1)+8\cos(2)\right);
$$
$$D(r)=40\cos(1)+24\sin(1)+24Si(1) -24Si(2) -4\cos(2)-6\sin(2)$$
$$-64 \cos(2)\cos^2(r)\sin^2(r)-32 \cos(1)\cos^2(r)$$
$$+12\left(  Si(2r-1) - Si(4r-2) +Si(4r+2) - Si(2r+1)\right)+18;  $$
$$G(r)=16\left( Si(4r-2)+Si(4r+2)-Si(2r-1) -Si(2r+1) \right).   $$
Plots of  $B_3(r),$  $b_3(r),$  and their normalized transformations are shown on Fig. 5 and Fig. 6.
\begin{center}
\begin{minipage}{5cm}
 {\psfig{figure=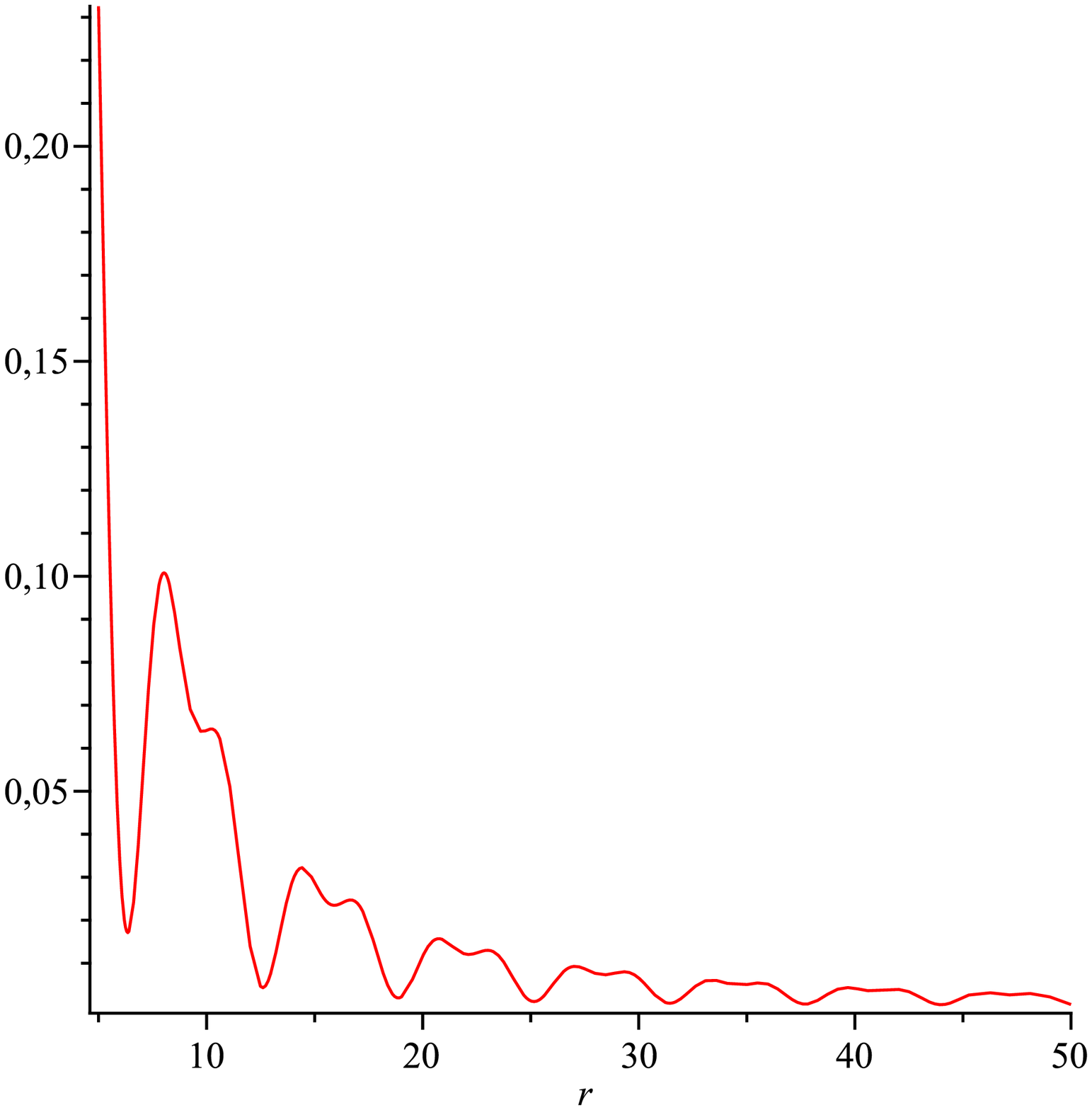 ,width=5cm}} \hspace{0.5cm}{  {\bf\rm Fig. 5a.} Plot of  $B_3(r)$} \end{minipage}
\begin{minipage}{5cm}
{\psfig{figure=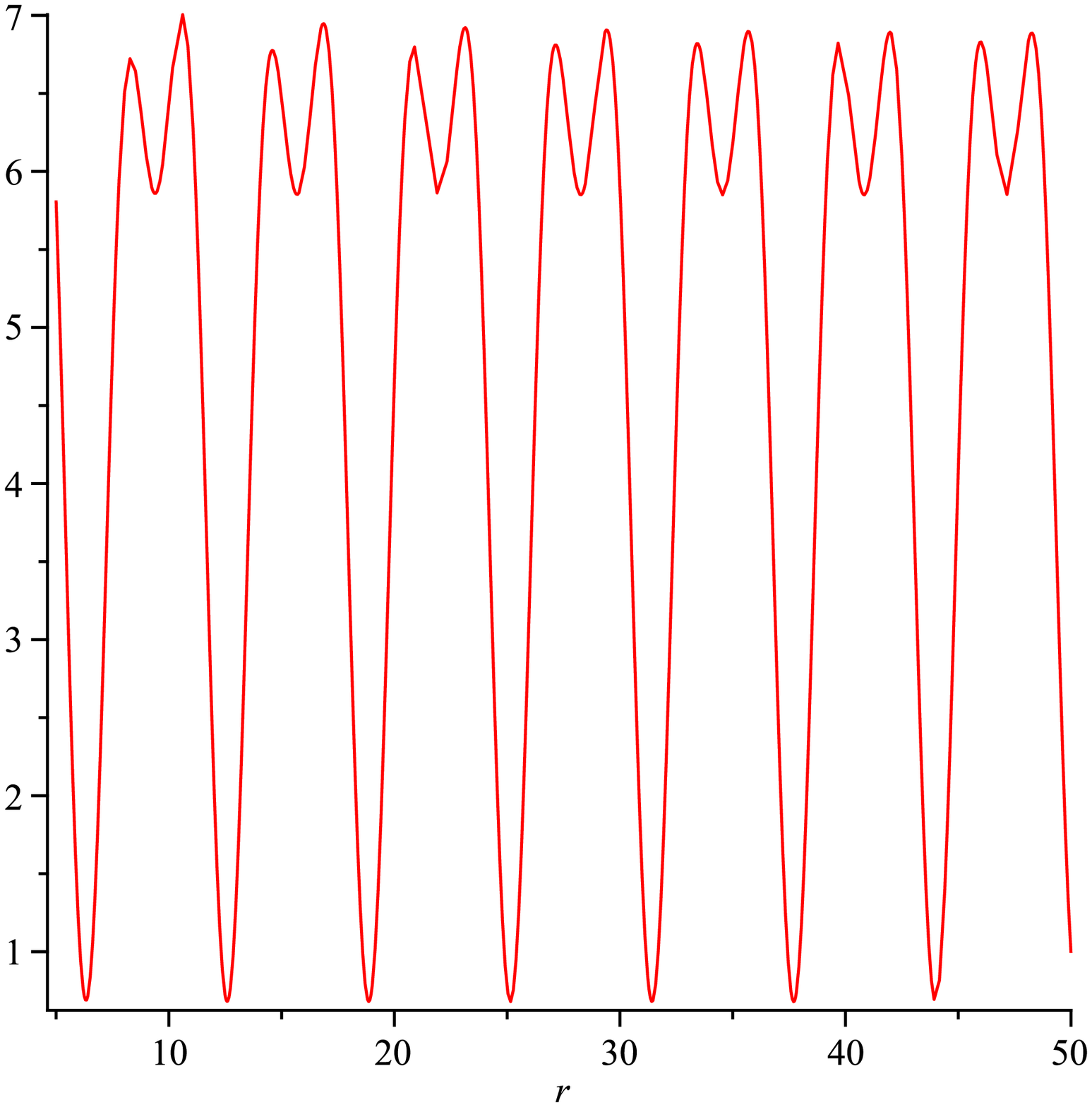,width=5cm}} \hspace{0.5cm}{  {\bf\rm Fig. 5b. } Plot of  $r^2 B_3(r)$} \end{minipage}
\end{center}
\begin{center}
\begin{minipage}{5cm}
 {\psfig{figure=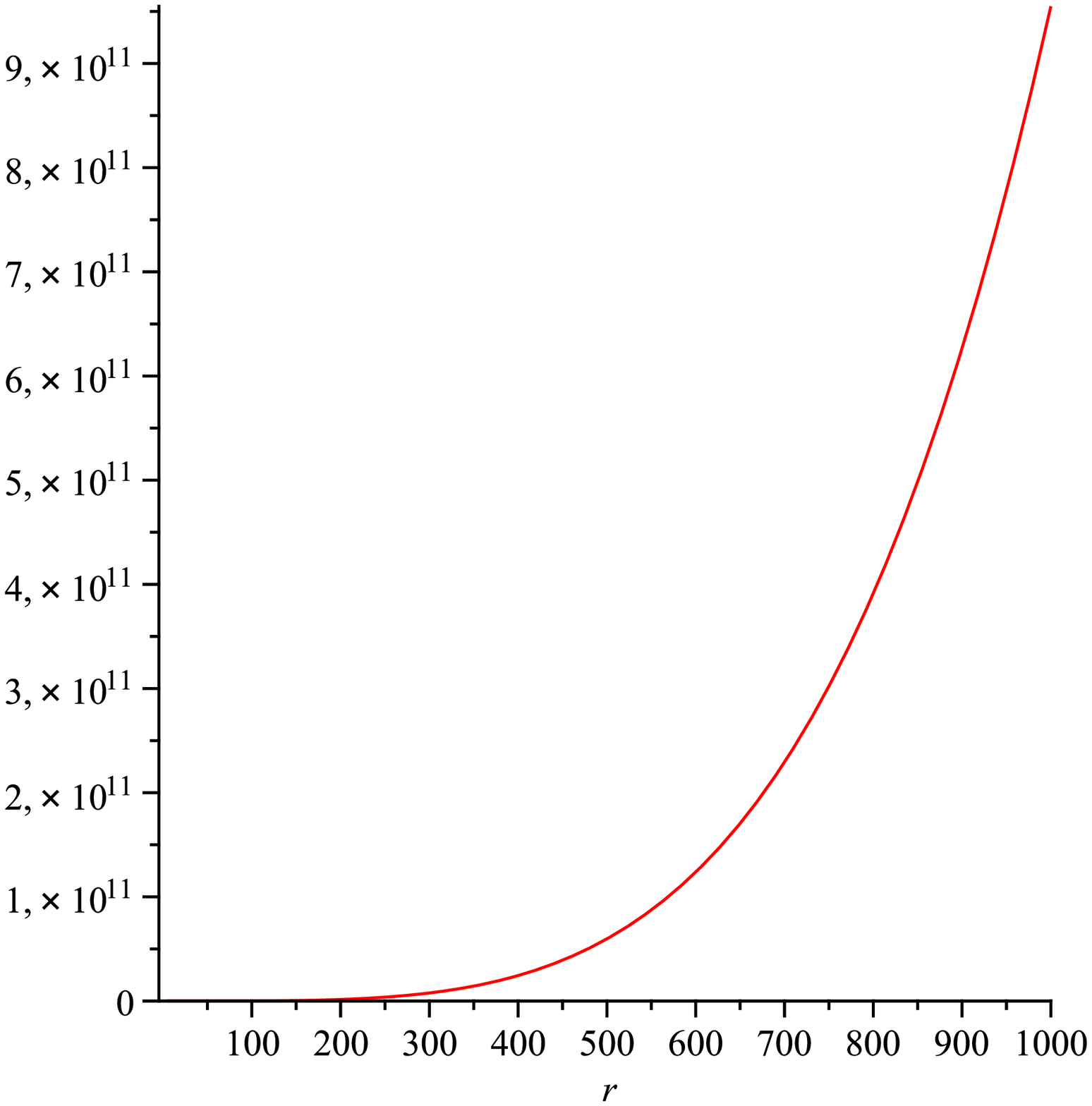 ,width=5cm}} \hspace{0.5cm}{  {\bf\rm Fig. 6a. } Plot of  $b_3(r)$} \end{minipage}
\begin{minipage}{5cm}
{\psfig{figure=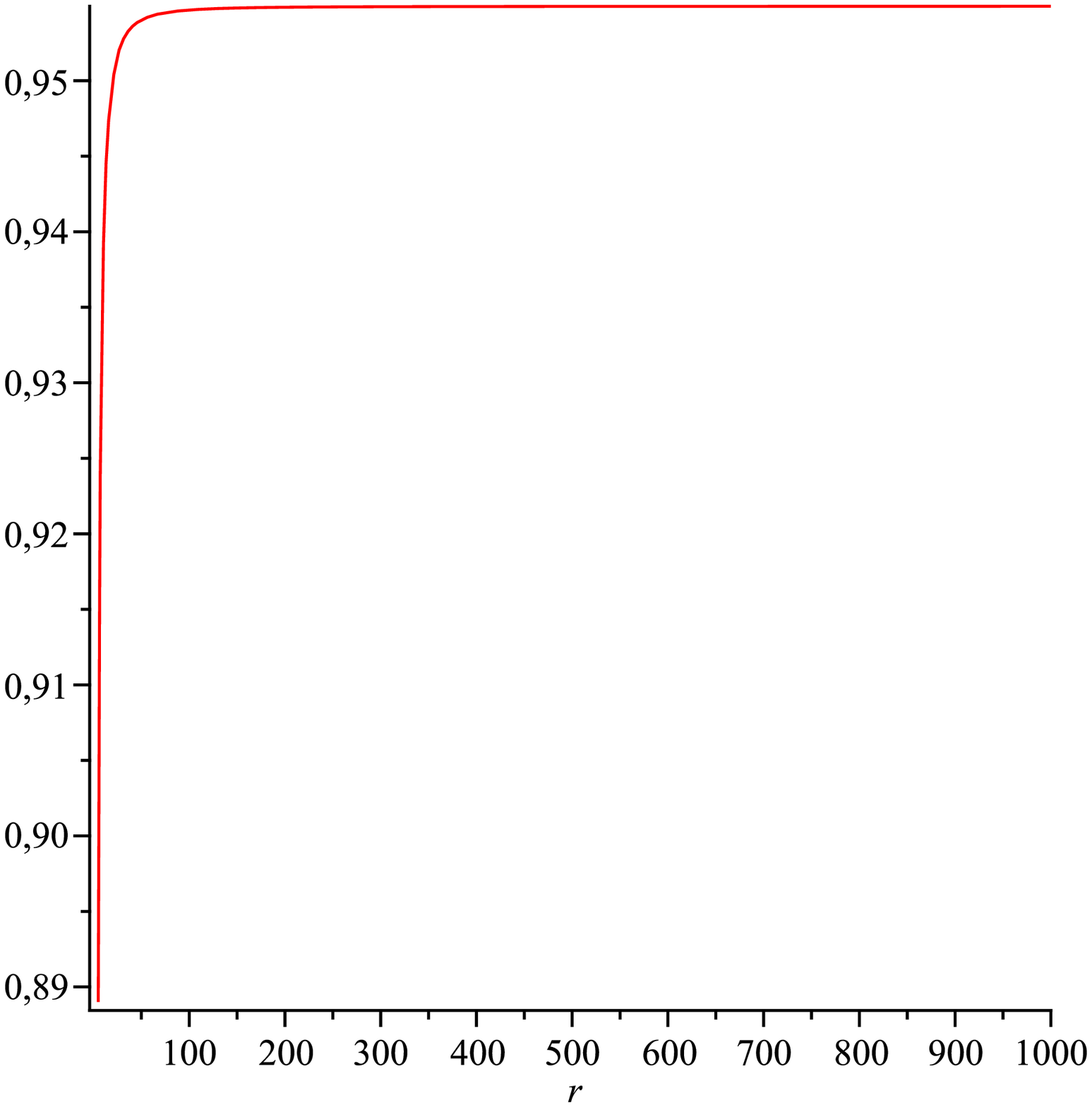,width=5cm}} \hspace{0.5cm}{  {\bf\rm Fig. 6b. } Plot of  $r^{-4} b_3(r)$} \end{minipage}
\end{center}
\end{example}

These new examples can find applications in spatial statistics.  Some empirical covariance functions demonstrate similar oscillatory behaviour.

In all examples above spectral densities have a singularity at zero, which is the long-range dependence case.   Functions $\varphi(\lambda)$ have the asymptote $\frac{1}{\lambda}$ when $\lambda \to +0, $
but corresponding covariance functions $B_3(r)$ are not regularly varying when $r\to \infty. $  Hovewer, functions $b_3(r)$  exhibit regular varying behaviour.

In the paper we investigate a more general class of spectral densities, which admit singularities at non-zero frequencies.  For these cyclical long-range dependent fields we study asymptotic behaviour of weighted functionals, which  generalize functionals (\ref{g001}).

\section{Spectral densities of cyclical long-range\\ dependent random fields}

Various properties of long-range dependent random fields were investigated in  \cite{leo1,leo2}.  The considered random fields had spectral singularities at zero and the spectral densities of the  form
  \begin{equation}\label{g1}
     {\varphi}(|\lambda{}|)=\frac{h_0(|\lambda{}|)}{|\lambda{}|^{n-\alpha{}}},
  \end{equation}
where $h_0(\cdot)$ is a bounded function defined on $\mathbb{R^+}:=[0,+\infty),$ and $ h_0(\cdot)$ is continuous in some neighborhood of
$0,$  $h_0(0)\ne 0.$

In this paper we study random fields with singularities of $\mathbf{\Phi}(\lambda{}) $ at the points $a_0,\dots,a_k\in [0,+\infty),$ $a_0=0,\ a_i<a_j,$ for $i<j.$

\begin{definition}  {\it  $\xi(x)$ is a cyclical long-range dependent random field if  it is  a  random field with the spectral density
\begin{equation}\label{g5}
    \varphi(|\lambda{}|)=\frac{h(|\lambda{}|)}{|\lambda{}|^{n-\alpha{}_0}}\prod_{i=1}^{k-1}\frac{1}{||\lambda{}|-a_i|^{1-\alpha{}_i}},
\end{equation}
where $\varphi(|\lambda{}|)\in L_1(\mathbb{R}^n),$ $0<\alpha{}_0<n,$ $0<\alpha{}_i<1,$ for $i=1,..,k-1, k\in \mathbb{N}.$ $h(\cdot)\ge 0$ is a bounded function defined on $[0,+\infty).$  $h(\cdot)$ is continuous in some neighborhoods of
 $a_i,$ and $h(a_i)\ne 0,\; i=0,...,k-1.$}
\end{definition}

 \begin{remark}
 Another additive models of spectral densities of seasonal/cyclic random fields and processes were considered in \cite{gir1, ole1, olkly, ole2}.
 \end{remark}

\begin{remark}
     For random processes $(n=1)$ the spectral density (\ref{g5}) has singularities at $k$ points. For random fields $(n\ge 2)$, the spectral density has singularities at all points of the $n$-dimensional sphere $S_n(a_i)=\{\lambda\in\mathbb{R}^n :|\lambda|=a_i\}$  when $a_i\ne 0.$
 \end{remark}

 \begin{remark}
   If $k=1$ then we obtain the class of long-range dependent random fields with  spectral densities given by {\rm(\ref{g1})}.
 \end{remark}


\section{Limit theorem }
In this section we use the approach and methods introduced in \cite{olkly, ole2}
to obtain limit theorems for random fields  with spectral densities given by  {\rm(\ref{g5})}.

 We  consider radial essentially non-zero weight functions $f_r(\cdot):\mathbb{R}^n \to \mathbb{R} $
 and $f_{a_j,r}(\cdot):\mathbb{R}^n \to\mathbb{R},$  $j=0,...,k-1,$  which satisfy the conditions:
   \begin{enumerate}
     \item $f_{a_j,r}(\cdot)\in L_m(\mathbb{R}^n),$ $m=1,2\,;$
     \item the Fourier transform of\, ${r^{-n}}f_{a_j,r}(\cdot)$ is a function of the form $g_{a_j}(r(|\lambda{}|-a_j)),$
     \item $g_{a_j}(s), s\in \mathbb{R}$ is an even function, and
\begin{equation}
    \label{g}\displaystyle g_{a_j}^2(s)\le \frac{C}{s},\quad \mbox{for all}\,  s> 0.
\end{equation}
   \end{enumerate}

Using approach  in  \cite{ole2} it is easy to show that if $f_r(\cdot)$ can be represented in the form
$    \label{f} f_r(x)=\sum_{j=0}^{k-1} C_j f_{a_j,r}(x),$
then this representation  is unique.

 \begin{theorem}\label{a2}
Let the isotropic spectral density $\varphi(\cdot)$ satisfy {\rm(\ref{g5})}. Then the finite dimensional distributions of the process
   $$X_{r,a_j}(t)= \left\{
               \begin{array}{ll}
{\displaystyle\frac{\sqrt{V_{a_j}}}{\sqrt{2\,h(a_j)}r^{n-\alpha_j/2}}\int_{\mathbb{R}^n}f_{a_j,rt}\left(x\right)\xi(x)dx,} & \ t\in(0,1];\vspace{1mm} \\
                 0, &\ t=0
               \end{array}
             \right. $$
 converge weakly to the finite dimensional distributions of the process
   $$ X_{a_j}(t)=\left\{
               \begin{array}{ll}
{\displaystyle{t^n}\int_{\mathbb{R}^n}\frac{g_{a_j}(|u|t)}{|u|^{\frac {n-\alpha{}_j} 2}}dZ(u)},& \ t\in(0,1];\vspace{1mm} \\
                 0, &\ t=0\,,
               \end{array}
             \right. $$
when $r\to\infty, \; j=1,...,k-1,$ $Z(\cdot)$ is the Wiener measure on $(\mathbb{R}^n, \mathfrak{B}^n),$ and
  $$  V_{a_j}=a_j^{1-\alpha_0}\prod_{ \scriptsize  \begin{array}{l}    i=1 \\  i\ne j   \end{array}}^{k-1}|a_j-a_i|^{1-\alpha_i}. $$
\end{theorem}

\begin{proof}  Let us define
\begin{equation*}
I_{a_j}(r):=\left\{
               \begin{array}{ll}
{\displaystyle{r^{-n}}\int_{\mathbb{R}^n}f_{a_j,r}\left(x\right)\xi(x)dx}, & \ r>0;\vspace{1mm} \\
                 0, &\ r=0\,.
               \end{array}
             \right.
\end{equation*}
The spectral representation of  the random field $\xi(x)$ is given by
\begin{equation*}
     \xi(x)=\int_{\mathbb{R}^n}e^{i<\lambda{},x>}\sqrt{\varphi(|\lambda{}|)}dW(\lambda{}),
\end{equation*}
where $W(\cdot) $ is the Wiener measure on  $(\mathbb{R}^n, \mathfrak{B}^n),$ see \cite{leo2}.

Using a stochastic analog of Plancherel's formula \cite{ole2} we obtain

\begin{gather} \label{g12j}
\begin{split}
      I_{a_j}(r) & = \,r^{-n} \int_{\mathbb{R}^n}\int_{\mathbb{R}^n}e^{i<\lambda{},x>}f_{a_j,r}\left({x}\right)dx\sqrt{\varphi(|\lambda{}|)}dW(\lambda{}) \\
    &\stackrel{d}{=}\int_{\mathbb{R}^n}g_{a_j}
(r(|\lambda{}|-a_j))\sqrt{\varphi(|\lambda{}|)}dW(\lambda{}).
\end{split}
\end{gather}

The integral on the right-hand  side of (\ref{g12j}) can be partitioned as follows
$$ \int_{\mathbb{R}^n} g_{a_j}(r(|\lambda{}|-a_j))\sqrt{\varphi(|\lambda{}|)}W(d\lambda{})
        \stackrel{d}{=} \int_{|\lambda{}|>a_j} g_{a_j}(r(|\lambda{}|-a_j))\sqrt{\varphi(|\lambda{}|)}dW(\lambda{})$$
$$+ \int_{|\lambda{}|<a_j} g_{a_j}(r(|\lambda{}|-a_j))\sqrt{\varphi(|\lambda{}|)}dW(\lambda{})=:K_1(r)+K_2(r).$$

First we deal with $K_1(r).$
By making the change of variables
$ u=\lambda{}\left(1-\frac{a_j}{|\lambda{}|}\right)$
and using results from \cite{olkly} and \cite{ole2} we get
     \begin{gather*}
      \begin{split}
         K_1(r) \stackrel{d}{=} &   \int_{\mathbb{R}^n} g_{a_j} (r|u|)\left(1+\frac{a_j}{|u|}\right)^{\frac{n-1}{2}} \\
          & \times                       \sqrt{\frac{h(|u|+a_j)}{\left(|u|+a_j\right)^{n-\alpha{}_0}}\prod_{i=1}^{k-1}\frac{1}{||u|-a_i+a_j|^{1-\alpha{}_i}}} d\widehat{W}(u),
    \end{split}
  \end{gather*}
where $\widehat{W}(\cdot)$ is the Wiener measure on $(\mathbb{R}^n, \mathfrak{B}^n).$

By  the  self-similarity of Gaussian white noise,
$d\widehat{W}(cx)\stackrel{d}{=}c^\frac{n}{2} d\widetilde{W}(x)\,,$ and making the change ov variables $\widetilde{u}=r u,$
we obtain
\begin{gather} \label{g15b1}  \begin{split} K_1(rt ) \stackrel{d}{=}& \int_{\mathbb{R}^n}  \frac{g_{a_j}(|\widetilde{u}|t)}{ \sqrt{r}|\widetilde{u}|^\frac {n-1}{2}}\left(a_j+ \frac{|\widetilde{u}|}{r}\right)^\frac{n-1}{2} \\&
\times \sqrt{\frac{h\left(\frac{|\widetilde{u}|}{r}+a_j\right)}{\left(\frac{|\widetilde{u}|}{r}+a_j  \right)^{n-\alpha_0}}
\prod_{i=1}^{k-1}\frac{1}{\left|\frac{|\widetilde{u}|}{r}-a_i+a_j\right|^{1-\alpha_i}}}\;d\widetilde{W}(\widetilde{u})
\end{split}
\end{gather}

Denote
   \begin{equation} \label{g16b}
       {Y}_{a_j,1}(t):=\left\{
               \begin{array}{ll}
{\displaystyle\int_{\mathbb{R}^n}\frac{g_{a_j}(|\widetilde{u}|t)}{|\widetilde{u}|^{\frac {n-\alpha_j}{2} }}d\widetilde{W}(\widetilde{u}),}
& \ t\in(0,1];\vspace{1mm} \\
                 0, &\ t=0\,.
               \end{array}
             \right.
   \end{equation}

By (\ref{g15b1}) and (\ref{g16b}),
\begin{gather}  \label{g17a} \begin{split}
      {R}_r(t)\,:=\,&\mathsf{E}\left( t^n \sqrt{\frac{V_{a_j} r^{\alpha{}_j}}{h(a_j)}}K_1(rt)-{t^n}{Y}_{a_j,1}(t)   \right)^2\\&
         = t^{2n}\int_{\mathbb{R}^n}\frac{g^2_{a_j}(|\widetilde{u}|t)}{|\widetilde{u}|^{n- \alpha{}_j}} {Q}_r(|\widetilde{u}|)d\widetilde{u},
\end{split}
\end{gather}
where
$${Q}_r(|\widetilde{u}|)\,:=\,\left(
\sqrt{\frac{V_{a_j} \cdot h\left(\frac{|\widetilde{u}|}{r} +a_j\right)}{h(a_j)\left(\frac{|\widetilde{u}|}{r} +a_j\right)^{1-\alpha_0}}
\prod_{ \scriptsize \begin{array}{l} i=1 \\  i\ne j \end{array}}^{k-1}\frac{1}{\left|\frac{|\widetilde{u}|}{r}-a_i+a_j\right|^{1-\alpha_i}}}
-1\right)^2.$$

Let us choose $1>\beta>\alpha_j,$
$$ B_1:=\{\widetilde{u}\in\mathbb{R}^n: |\widetilde{u}|\le r^{\beta}\} \quad \mbox{and}\quad
B_2:=\{\widetilde{u}\in\mathbb{R}^n: |\widetilde{u}|> r^{\beta}\}. $$

Then $\mathbb{R}^n=B_1\cup B_2$ and the integral on the right-hand side of (\ref{g17a}) becomes the sum $R_r(t)=R_{r,1}(t)+R_{r,2}(t).$

By $L_1$ and $L_2$-integrability of the function $f_{a_j,r}(\cdot)$ we obtain that  $g_{a_j}^2(|\lambda{}|)\in L_2(\mathbb{R}^n)$ and $g_{a_j}^2(|\lambda{}|)$ is bounded. Hence
the integral
$$ t^{2n}\int_{\mathbb{R}^n}\frac{g_{a_j}^2
(|\widetilde{u}|t)}{|\widetilde{u}|^{n- \alpha{}_j  }} d\widetilde{u}=t^{2n-\alpha{}_j}\int_{\mathbb{R}^n}\frac{g_{a_j}^2
(|\widetilde{z}|)}{|\widetilde{z}|^{n- \alpha{}_j}} d\widetilde{z},
$$
is uniformly bounded on $t\in[0,1].$
By {\rm(\ref{g5})}, given any $\varepsilon>0,$ there is  $r_0>0$ such that $Q_r(|\widetilde{u}|)<\varepsilon,$ when $r>r_0,\;\widetilde{u}\in B_1.$
 Thus
$R_{r,1}(t)$ can be made arbitrarily small  by decreasing the value of $\varepsilon.$

The decay rate of $g_{a_j}^2(\cdot)$ implies that there is $r^0>0 $ such that, for any $r>r^0,$
\begin{equation*}
 \begin{split}  {R}_{r,2}(t) & \le  2t^{2n} {\Huge\textrm{(}} \int_{B_2}\frac{g_{a_j}^2
(|\widetilde{u}|t)}{|\widetilde{u}|^{n- \alpha{}_j }}d\widetilde{u}
   +\int_{B_2}\frac{g_{a_j}^2
(|\widetilde{u}|t)}{|\widetilde{u}|^{n- \alpha{}_j }}
\frac{V_{a_j}}{h(a_j)} \frac{h\left(\frac{|\widetilde{u}|}{r}+a_j\right)}{\left( \frac{|\widetilde{u}|}{r}+a_j\right)^{1-\alpha_0}} \\
&\times
\prod_{ \scriptsize \begin{array}{l} i=1 \\  i\ne j \end{array}}^{k-1}\frac{1}{\left|\frac{|\widetilde{u}|}{r}-a_i+a_j\right|^{1-\alpha_i}}d\widetilde{u}  {\Huge\textrm{)}}
 \le C_1t^{2n-1} \int_{B_2}\frac{d\widetilde{u}}{|\widetilde{u}|^{n+1-\alpha_j}} \\ &
 +C_2 \frac{ t^{2n-1}}{r^{\beta-\alpha_j}}\int_{r^{\beta-1}}^{+\infty}\frac{h\left(\rho+a_j\right)}{\left(\rho+a_j\right)^{1-\alpha_0}}
\prod_{i=1}^{k-1} \frac{1}{\left|\rho-a_i+a_j\right|^{1-\alpha_i}}d\rho.
\end{split}
\end{equation*}

Now, by  the choice of $\beta>\alpha_j$ and $\varphi(|\lambda{}|)\in L_1(\mathbb{R}^n)$ the upper bound vanishes when   $r\to+\infty.\;\;$
Thus, $\lim_{r\to\infty}{R}_r(t)=0$    and the   finite dimensional distributions of the process $t^n\sqrt{{V_{a_j}r^{\alpha{}_j}}/{h(a_j)}}\,K_1(rt)$ converge to the finite dimensional distributions of the process $t^n\,{Y}_{a_j,1}(t),$ when $r \to +\infty. $

Similar to the case of $K_1(\cdot)$ by the change of variables $u=\lambda\left(\frac{a_j}{|\lambda|}-1\right)$  (see \cite{olkly, ole2}) we obtain

\begin{gather} \label{g15b}
\begin{split}  K_2(rt ) \stackrel{d}{=}& \int_{|\widetilde{u}|<r a_j}  \frac{g_{a_j}(|\widetilde{u}|t)}{ \sqrt{r}|\widetilde{u}|^\frac {n-1}{2}}
\left(a_j- \frac{|\widetilde{u}|}{r}\right)^\frac{n-1}{2}  \\& \times \sqrt{\frac{h\left(a_j-\frac{|\widetilde{u}|}{r}\right)}{\left(a_j-\frac{|\widetilde{u}|}{r}  \right)^{n-\alpha_0}}
\prod_{i=1}^{k-1}\frac{1}{\left|a_j-\frac{|\widetilde{u}|}{r}-a_i\right|^{1-\alpha_i}}}\;d\overline{W}(\widetilde{u}),
\end{split}
\end{gather}
where $\overline{W}(\cdot)$ is a Wiener measure on $(\mathbb{R}^n, \mathfrak{B}^n).$

Denote
   \begin{equation} \label{g16b-}
       Y_{a_j,2}(t):=\left\{
               \begin{array}{ll}
{\displaystyle \int_{\mathbb{R}^n}\frac{g_{a_j}(|\widetilde{u}|t)}{|\widetilde{u}|^{\frac {n-\alpha{}_j}{2} }}d\overline{W}(\widetilde{u}),}
& \ t\in(0,1];\vspace{1mm} \\
                 0, &\ t=0\,.
               \end{array}
             \right.
   \end{equation}
By (\ref{g15b}) and (\ref{g16b-}), we obtain
$${S}_r(t):=\mathsf{E}\left( t^{n} \sqrt{\frac{V_{a_j}{}r^{\alpha{}_j}}{h(a_j)}}K_2(rt)-t^{n}{Y}_{a_j,2}(t)   \right)^2$$
\begin{equation} \label{g17b-} = t^{2n}\int_{\mathbb{R}^n}\frac{g^2_{a_j}(|\widetilde{u}|t)}{|\widetilde{u}|^{n- \alpha{}_j}} \overline{Q}_r(|\widetilde{u}|)d\widetilde{u},
\end{equation}
where
$$\overline{Q}_r(|\widetilde{u}|):=
\left\{
               \begin{array}{ll}
\left(\sqrt{\frac{h\left(a_j-\frac{|\widetilde{u}|}{r}\right)}{h(a_j)}
{\displaystyle\prod_{ \scriptsize \begin{array}{l} i=0 \\  i\ne j \end{array}}^{k-1}}\frac{V_{a_j}}{\left|a_j-\frac{|\widetilde{u}|}{r}-a_i\right|^{1-\alpha_i}}}
-1\right)^2,\; & \mbox{if}\;\; a_j r>|\widetilde{u}|, \\
         0,\; & \mbox{if} \;\;a_j r\le |\widetilde{u}|.
   \end{array}%
           \right.
$$

The integral on the right-hand side of (\ref{g17b-}) can be split into two parts
${S}_r(t)={S}_{r,1}(t)+{S}_{r,2}(t),$ where the
integration sets are respectively $B_1$ and $B_2.$    Similar to the case of $R_{r,1}(t)$  the integral
$S_{r,1}(t)$  can be made arbitrarily small.

By (\ref{g}) there is $r^0>0,$ such that, for any $r>r^0,$
 $$  {S}_{r,2}(t) \le  C_1 t^{2n-1} \int_{B_2}\frac{d\widetilde{u}}{|\widetilde{u}|^{n+ 1 - \alpha{}_j }}    +C_2 \frac{ t^{2n-1}}{r^{\beta-\alpha_j}}$$ $$ \times \int_{r^{\beta-1}}^{a_j}\frac{h(a_j-\rho)}{(a_j-\rho)^{1-\alpha_0}}
\prod_{i=1}^{k-1} \frac{1}{\left|a_j-\rho-a_i\right|^{1-\alpha_i}}d\rho,
$$
and the upper bound vanishes when $r\to+\infty.$

Therefore, \,the \, finite\,  dimensional\, distributions of \, the\,  process  \;$t^n\sqrt{{V_{a_j}r^{\alpha{}_j}}/{h(a_j)}}\,K_2(rt)$\
converge to the finite dimensional distributions of the process ${t^n}\,Y_{a_j,2}(t),\; t\in[0,1],$ when $r\to+\infty.$

The above results imply the convergence of the finite dimensional distributions of the process
$$ \sqrt{2}\, X_{r,a_j}(t)\stackrel{d}{=} t^n\sqrt{\frac{V_{a_j}r^{\alpha{}_j}}{h(a_j)}}\cdot\Big(K_1(rt)+K_2(rt)\Big) $$
to the finite dimensional distributions of ${t^n}  ( Y_{a_j,1}(t)+\,Y_{a_j,2}(t)), \;\;t\in[0,1].$

Finally,  since the Winer measures $\widetilde{W}(\cdot)$ and $\overline{W}(\cdot)$ are independent, it follows that there exists the Winer measure $Z(\cdot) $ on $(\mathbb{R}^n, \mathfrak{B}^n)$
such that, formally, $ \sqrt{2}dZ(\cdot)\stackrel{d}{=} d\widetilde{W}(\cdot)+d\overline{W}(\cdot ),$
and
\begin{eqnarray*}
       {t^n} ( Y_{a_j,1}(t)+ Y_{a_j,2}(t)) \stackrel{d}{=}
      \sqrt{2} t^n \int_{\mathbb{R}^n}\frac{g_{a_j}(|\widetilde{u}|t)}{|\widetilde{u}|^{\frac {n-\alpha{}_j} 2}}dZ(\widetilde{u})=\sqrt{2}\,X_{a_j}(t).
\end{eqnarray*}
\end{proof}

\section{Degenerated limits}

Theorem \ref{a2}  gives the limit behaviour of the functional $X_{r,a_j}(t)$ of the random field $\xi(x)$ with unbounded spectral density at $a_j.$
To complete this research we investigate those cases where the random field $\xi(x)$ in the integral $X_{r,a}(t)$ does not have singularities at $a.$

\begin{theorem}  \label{a3}
  Let the isotropic spectral density $\varphi(|\lambda{}|)$ be of the form \eqref{g5} and $a\notin \{a_0,\,...\,,a_{k-1}\}.$  Then, for any normalization ${1}/{r^{n-\alpha/2}}, \, 0<\alpha<1,$ the finite dimensional distributions of the process $X_{r,a}(t)$ converge to $0,$ when $r\to +\infty.$
\end{theorem}
\begin{proof}
By the proof of  Theorem~\ref{a2} it is enough to show that
$$    \lim_{r \to \infty} r^{\alpha }t^{2n} \mathsf{E}\left[ K^2_1(rt) \right] =   \lim_{r \to \infty} r^{\alpha}t^{2n} \mathsf{E}\left[ K^2_2(rt) \right] =0. $$
It is easy to see that
$$ \widetilde{R}_r(t):=r^{\alpha}t^{2n} \mathsf{E}\left[ K^2_1(rt) \right] =  \frac{t^{2n}}{r^{1-\alpha}}\int_{\mathbb{R}^n}\frac{g^2_a(|\widetilde{u}|t)}{|\widetilde{u}|^{n- 1}} {\widetilde{Q}}_r(|\widetilde{u}|)d\widetilde{u},  $$
where
$$\widetilde{Q}_r(|\widetilde{u}|)\,:=\,
\frac{h\left(\frac{|\widetilde{u}|}{r} +a\right)}{\left(\frac{|\widetilde{u}|}{r} +a\right)^{1-\alpha_0}}
\prod_{  i=1}^{k-1}\frac{1}{\left|\frac{|\widetilde{u}|}{r}-a_i+a\right|^{1-\alpha_i}}.
$$
Using $B_1$ and $B_2$ from Theorem \ref{a2} and $1>\beta>\alpha$ we split $\widetilde{R}_r(t)$
into the sum $\widetilde{R}_r(t)=\widetilde{R}_{r,1}(t)+\widetilde{R}_{r,2}(t).$  There is $\widetilde{r}_0>0$  such that $|\widetilde{Q}_r(|\widetilde{u}|)|<C,$
where $r>\widetilde{r}_0,\, \widetilde{u}\in B_1. $ Using boundness of $g_{a}^2(|\lambda|)$ and  \eqref{g} we obtain
\begin{gather*}  \begin{split} \widetilde{R}_{r,1}(t) & \; \le \frac{t^{2n}}{r^{1-\alpha}}\left(\int_{t|\widetilde{u}|<1}\frac{C}{|\widetilde{u}|^{n- 1}} d\widetilde{u}+C_1\int_{1\le t|\widetilde{u}|\le t r^{\beta}}
\frac{g^2_a(|\widetilde{u}|t)}{|\widetilde{u}|^{n- 1}} d\widetilde{u} \right)\\
& \;\; = \frac{C\, t^{2n-1}}{r^{1-\alpha}} + \frac{C_2 t^{2n-1}}{r^{1-\alpha}}\int_{t^{-1}\le |\widetilde{u}|\le  r^{\beta}}
\frac{d\widetilde{u}}{|\widetilde{u}|^{n}} =  \frac{C \,t^{2n-1}}{r^{1-\alpha}}  \\
& \;\;\; +\frac{C_2 t^{2n-1}}{r^{1-\alpha}}\left(ln\left(r^{\beta}\right)+ln(t)\right)  \to 0 \, ,
\end{split}
\end{gather*}
when $r\to+\infty.$

By  \eqref{g}  one can obtain that there is $\widetilde{r}^0>0$  such that, for any  $r>\widetilde{r}^0:$
$$ \widetilde{R}_{r,2}(t)\le \frac{C\, t^{2n-1}}{r^{\beta-\alpha}}\int_{r^{\beta-1}}^{+\infty}\frac{h\left(\rho+a\right)}{\left(\rho+a\right)^{1-\alpha_0}}
\prod_{ i=1 }^{k-1} \frac{1}{\left|\rho-a_i+a\right|^{1-\alpha_i}}d\rho.
$$

Now, because of the choice $\beta>\alpha$ we get that the upper bound vanishes when $r\to+\infty. $
Thus $\lim_{r\to \infty}\widetilde{R}_r(t)=0. $

Similar to \eqref{g17b-}  we obtain

\begin{equation*} \begin{split} &
      {\widetilde{S}}_r(t):=r^{\alpha}t^{2n}\mathsf{E}\left(K^2_2(rt)\right)=\frac{t^{2n}}{r^{1-\alpha}}
         \int_{\mathbb{R}^n}\frac{g^2_a(|\widetilde{u}|t)}{|\widetilde{u}|^{n- 1}} \hat{Q}_r(|\widetilde{u}|)d\widetilde{u},
\end{split}
\end{equation*}
where
$$\hat{Q}_r(|\widetilde{u}|):=
\left\{
               \begin{array}{ll}
\frac{h\left(a-\frac{|\widetilde{u}|}{r}\right)}{\left(a-\frac{|\widetilde{u}|}{r}\right)^{1-\alpha_0}}
{\displaystyle\prod_{  i=1 }^{k-1}}\frac{1}{\left|a-\frac{|\widetilde{u}|}{r}-a_i\right|^{1-\alpha_i}}.
,\; & \mbox{if}\;\; a r>|\widetilde{u}|, \\
         0,\; & \mbox{if} \;\;a r\le |\widetilde{u}|.
   \end{array}%
           \right.
$$

We split the integral $\widetilde{S}_r(t)$  into two parts
$\widetilde{S}_r(t)=\widetilde{S}_{r,1}(t)+\widetilde{S}_{r,2}(t),$ where the
integration sets are respectively $B_1$ and $B_2.$    Similar to the case of $\tilde{R}_{r,1}(t)$  the integral
$\widetilde{S}_{r,1}(t)$  can be made arbitrarily small  by increasing the value of  $r.$

By \eqref{g} there is $\widetilde{r}^0>0,$  such that, for any  $r>\widetilde{r}^0,$
\begin{gather*}  \begin{split}
     & \widetilde{S}_{r,2}(t)\le\frac{C t^{2n-1}}{r^{\beta-\alpha}}\int_{r^{\beta-1}}^{\alpha}
     \frac{h\left(a-\rho\right)}{\left(a-\rho\right)^{1-\alpha_0}}
\prod_{ i=1 }^{k-1} \frac{1}{\left|a-\rho-a_i\right|^{1-\alpha_i}}d\rho.
\end{split}
\end{gather*}
The upper bound vanishes when $r\to +\infty,$  thus, $\lim_{r\to+\infty}\widetilde{S}_{r}(t)=0,$ which completes the proof.

\end{proof}

\begin{remark}
    If in Theorem \ref{a3} $\alpha=1$ then, similar to the one-dimen\-sional case, we would expect that the central limit theorem holds.
\end{remark}

\section{Final remarks }

1.  A similar limit theorem for cyclical long-range dependent random fields was proved in \cite{ole2}. Namely,  an additive class of singular spectral densities was investigated  instead of \eqref{g5}. It is easy to show that the additive class of spectral densities does not coincide with the multiplicative class considered in this paper.  Different normalization constants in Theorem~\ref{a2} and the decay condition \eqref{g} were investigated in this paper.  Degenerated limits were not studied in \cite{ole2}.\\
2.  Similar results for a specific choice of weight functions and additive spectral densities were obtained in \cite{gir1}. Namely,
the discrete case and the functionals
$$ X_{N,\beta}(t)=\frac{1}{A_N}\sum_{m=0}^{[Nt]}e^{-im\beta}H_k(\xi_m) $$  with the particular weight functions
$$ f_{\beta, Nt}(m) = \left\{
\begin{array}{ll} e^{-i m \beta}, &  0\le m\le [Nt]; \\
0, & \textrm{otherwise} \\
\end{array}\right.
$$
were investigated.\\
3.  A particular example of random fields with only singularity at the frequency $a_1\not= 0$ and the specific weight function
$f_{a_1,r}(\cdot)$ with the Fourier transform
$$g_{a_1}(r(|\lambda{}|-a_1))=(2\pi)^{n/2}\frac{J_{\frac{n}{2}}(r(|\lambda{}|-a_1))}{(r(|\lambda{}|-a_1))^{n/2}}$$
was studied in \cite{ole1} and \cite{olkly}.\\
4.  One can easily obtain new examples of the weight functions $f_{a,r}(x)$ and $g_{a}(r(|\lambda|-a)),$  by using the representation of the Fourier transform of radial functions
\cite{leo1}
\begin{equation*}f_{a,r}(x)=\frac{1}{(2\pi)^{n/2}|x|^{\frac{n}{2}-1}}\int_0^\infty\mu^{n/2}
g_{a}(r(\mu-a))J_{\frac{n}{2}-1}(|x|\mu)d\mu,
\end{equation*}
and tables of the Hankel transforms.\\
5. It would be of interest to prove limit theorems for functionals of the form
       $$\int_{\mathbb{R}^n}f_{a_j,rt}(x)H_m(\xi(x))dx,\, m>1,$$
where $\xi(x)$ exhibits cyclical long-range dependence.\\

\bigskip

CONTACT INFORMATION

\medskip
B.~Klykavka\\ Department of Mathematical Analysis and Probability Theory, National Technical University of Ukraine «KPI», 03056 Kyiv, Ukraine  \\ bklykavka@yahoo.com

\medskip
A.~Olenko \\ Department of Mathematics and Statistics, La Trobe University,\\ Victoria, 3086, Australia \\ A.Olenko@latrobe.edu.au

\medskip
M.~Vicendese\\ Department of Mathematics and Statistics, La Trobe University,\\ Victoria, 3086, Australia \\ fishardansin@gmail.com

\end{document}